\documentclass[a4paper,12pt]{article}
    \usepackage{amsmath}
    \usepackage{amssymb}
    \usepackage{amsthm}

\usepackage{graphicx,psfrag}
\usepackage{pstricks}
\usepackage{subfigure}

      \newcommand {\del}  {\delta}          
              \newcommand {\ve}   {\varepsilon}
                 \newcommand {\vphi} {\varphi}

                \newcommand {\Om}  {\Omega}
      \newcommand {\pl}   {\partial}        
           
      \newcommand {\RRR}  {{\mathbb R}}     
              
      \newcommand {\ZZZ}  {{\mathbb Z}}

      \newcommand {\bbss}  {\begin{slide}}
      \newcommand {\eess}  {\end{slide}}
      \newcommand {\beq}  {\begin{equation}}
      \newcommand {\eeq}  {\end{equation}}
      \newcommand {\trap}{\text{\pspolygon[linewidth=0.4pt](-0.02,0)(0.32,0)(0.23,0.22)(0.07,0.22)\ \ \,}}

      \newtheorem{theorem}{Theorem}

       \newtheorem{remark}{Remark}
      
      \newtheorem{propo}{Proposition}

\author{Alexander Plakhov\thanks{Center for Research and Development in Mathematics and Applications, Department of Mathematics, University of Aveiro, Portugal and Institute for Information Transmission Problems, Moscow, Russia}}

\title{The problem of minimal resistance for functions and domains}

\date{}

\begin{document}

\maketitle

\begin{center}
To the memory of T. Lachand-Robert.
\end{center}

\begin{abstract}
Here we solve the problem posed by Comte and Lachand-Robert in \cite{CL3} (2001). Take a bounded domain $\Om \subset \RRR^2$ and a piecewise smooth non-positive function $u : \bar\Om \to \RRR$ vanishing on $\pl\Om$. Consider a flow of point particles falling vertically down and reflected elastically from the graph of $u$. It is assumed that each particle is reflected no more than once (no multiple reflections are allowed); then the resistance of the graph to the flow is expressed as $R(u;\Om) = \frac{1}{|\Om|} \int_\Om (1 + |\nabla u(x)|^2)^{-1} dx$. It is required to find $\inf_{\Om,u} R(u;\Om)$. One can easily see that $|\nabla u(x)| < 1$ for all regular $x \in \Om$, and therefore one always has $R(u;\Om) > 1/2$.

We prove that the infimum of $R$ is exactly $1/2$. This result is somewhat paradoxical, and the proof is inspired by, and is partly similar to, the paradoxical solution given by Besicovitch to the Kakeya problem \cite{Bes}.
\end{abstract}

\begin{quote}
\end{quote}

\begin{quote}
\end{quote}

\section{Introduction}

The famous problem of least resistance first stated by Newton in \cite{N} gave rise to a series of interesting variational problems that have been intensively studied in the last two decades (see, e.g., \cite{BrFK}-\cite{LP1}). These problems originate from a simple mechanical model where a flow of point particles with constant velocity is incident on a solid body. When hitting the surface of the body, the particles are elastically reflected from it. The flow is homogeneous and is so rarefied that mutual interaction of particles can be neglected. The force acting on the body and created by collisions of the flow particles with the body is called the {\it resistance}. One needs to find the body, from a prescribed class of bodies, that minimizes the resistance.

One normally imposes an additional condition on the body shape stating that no particle collides with the body more than once. With this assumption, the problem of minimal resistance can be written in a comfortable analytic form. Namely, introduce orthogonal coordinates $x_1,\, x_2,\, z$, assume that the flow falls vertically down with the velocity $(0,0,-1)$, and define the function $u : \Om \to \RRR$ whose graph coincides with the upper part of the body. (Here $\Om \subset \RRR^2$ is the orthogonal projection of the body on the $(x_1,x_2)$-plane and the graph of $u$ is formed by points of collision with the flow particles.) Then the resistance of the body equals $2\rho |\Om|\, R(u;\Om)$, where $\rho$ is the density of the flow, $|\Om|$ is the area of $\Om$, and
\beq\label{o Resistance}
R(u;\Om) = \frac{1}{|\Om|} \int_\Om \frac{dx}{1 + |\nabla u(x)|^2},
\eeq
with $x = (x_1,x_2)$ being a point of the plane.

On the other hand, the condition itself (in what follows it will be called {\it single impact condition}, or just SIC) has a quite complicated analytic form and is difficult to deal with. Therefore it is often substituted with some other conditions, which are stronger but easier to work with \cite{BrFK,BFK}. One such condition which is very popular in the literature states that the body is convex (and therefore the corresponding function $u$ is concave) \cite{BrFK,BFK,BK,LO,LP1}.

A very interesting problem was proposed in 2001 by Comte and Lachand-Robert \cite{CL3}. Let $\Om \subset \RRR^2$ be a bounded domain and let $u : \bar\Om \to \RRR$ be a piecewise $C^1$ function such that $u\rfloor_{\pl\Om} = 0$ and $u(x) < 0$ for all $x \in \Om$. It is additionally assumed that $u$ satisfies the single impact assumption which can be stated analytically as follows: for any regular point $x \in \Om$ and any $t > 0$ such that $x - t\nabla u(x) \in \bar\Om$,
\beq\label{o SIC}
\frac{u(x - t\nabla u(x)) - u(x)}{t} \le \frac{1}{2} (1 - |\nabla u(x)|^2).
\eeq
A function $u$ satisfying the above assumptions will be called {\it admissible}.

\begin{remark}\label{z1}\rm
It can be easily seen that condition (\ref{o SIC}) is equivalent to the following geometric condition: any particle of the vertical flow with the velocity $(0,0,-1)$, after the perfectly elastic reflection from a regular point of the graph of $u$, further moves freely above the graph (it may, however, touch the graph at some points).
\end{remark}

In Fig. \ref{fig_sic} examples of function satisfying and not satisfying SIC are given.

\begin{figure}[h]
\begin{picture}(0,120)
\rput(3.3,2){
\psellipse(0,0)(2,0.25)
\pscurve(-2,0)(-1.5,-0.4)(-1,-0.7)(-0.5,-0.9)(0,-1)(0.5,-0.88)(1,-0.77)(1.5,-0.37)(2,0)
\psline[linewidth=0.5pt,linecolor=red,arrows=->,arrowscale=1.5](-1.7,1.5)(-1.7,-0.1)
\psline[linewidth=0.2pt,linecolor=red](-1.7,-0.1)(-1.7,-0.2)(-1.38,-0.13)
\psline[linewidth=0.5pt,linecolor=red,arrows=->,arrowscale=1.5](-1.38,-0.13)(1.5,0.5)
\psline[linewidth=0.5pt,linecolor=red,arrows=->,arrowscale=1.5](0.5,1.5)(0.5,-0.2)
\psline[linewidth=0.2pt,linecolor=red](0.5,-0.2)(0.5,-0.88)(0.06,-0.22)
\psline[linewidth=0.5pt,linecolor=red,arrows=->,arrowscale=1.5](0.06,-0.22)(-0.4,0.47)
\rput(-2.2,-1.3){(a)}
}
\rput(10.6,2){
\psellipse(0,0)(1.9,0.35)
\pscurve(-1.9,0)(-1.5,-0.6)(-1,-1.04)(-0.5,-1.3)(0,-1.45)(0.5,-1.27)(1,-1)(1.5,-0.57)(1.9,0)
\psline[linewidth=0.5pt,linecolor=red,arrows=->,arrowscale=1.5](-1.5,1.5)(-1.5,-0.2)
\psline[linewidth=0.2pt,linecolor=red,arrows=->,arrowscale=1.5](-1.5,-0.2)(-1.5,-0.6)(-0.27,-0.8)
\psline[linewidth=0.2pt,linecolor=red](-0.27,-0.8)(0.96,-1)(1.05,-0.25)
\psline[linewidth=0.5pt,linecolor=red,arrows=->,arrowscale=1.5](1.05,-0.25)(1.26,1.5)
\rput(-2.2,-1.3){(b)}
}
\end{picture}
\caption{Examples of a function that (a) satisfies and (b) does not satisfy single impact condition.}
\label{fig_sic}
\end{figure}
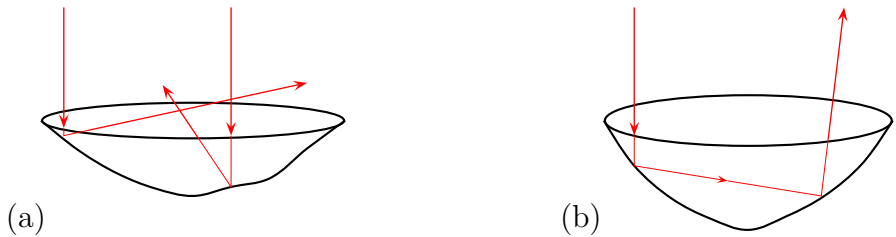

The problem stated by Comte and Lachand-Robert reads as follows.
\vspace{2mm}

%\begin{problem}\label{Problem1}
{\bf Problem.} {\it Minimize the functional $R(u;\Om)$ (\ref{o Resistance}) over all bounded domains $\Om$ and admissible functions $u : \bar\Om \to \RRR$.}
%\end{problem}
\vspace{2mm}

\begin{remark}\label{z2}\rm
The original formulation in \cite{CL3} concerns the (seemingly more restricted) problem of minimization over bounded domains $\Om$ tiling the plane. However, as follows from Theorem \ref{t2}, these problems are equivalent.
\end{remark}

Two other interesting problems of minimal resistance for bodies satisfying the single impact condition, with and without rotational symmetry, were studied by Comte and Lachand-Robert in \cite{CL1} and \cite{CL2}.

Curiously, the problem of Comte and Lachand-Robert admits the following mechanical interpretation. An aircraft moves at a very large height in a thin atmosphere. One wants to make small dimples on parts of wings or fuselage so as to diminish the aerodynamic resistance of the aircraft. The problem amounts to optimization of the shape of dimples. Of course, the mechanical assumptions adopted here are oversimplified, especially as concerns perfectly elastic reflections of atmospheric particles from the aircraft.

It is easy to see that for any domain $\Om$,\, $\sup_u R(u;\Om) = 1$, and the supremum is attained, as $n \to \infty$, at any sequence of functions of the form $\frac 1n u(x)$. On the other hand, for any admissible function $u$ and any regular point $x \in \Om$ we have
\beq\label{nabla u}
|\nabla u(x)| < 1.
\eeq
This can be derived from both geometric considerations and formula (\ref{o SIC}). Indeed, if $|\nabla u(x)| \ge 1$, the particle reflected from the point $(x,u(x))$ will then move downward (the third component of velocity will be non-positive), and therefore will inevitably hit the graph of $u$ once again.

One can also use analytical reasoning: when $x - t\nabla u(x) \in \pl\Om$, the left hand side of formula (\ref{o SIC}) is positive, and therefore, the right hand side should also be positive.

As a result, one always has $R(u;\Om) > 1/2$, and so, for any $\Om$
$$
\sup_u R(u;\Om) = 1 \quad \text{and} \quad \inf_u R(u;\Om) \ge 1/2.
$$

\begin{figure}[h]
\centering
%\rput(0,0)
\vspace*{15mm}
{\includegraphics[width=190pt,height=120pt]{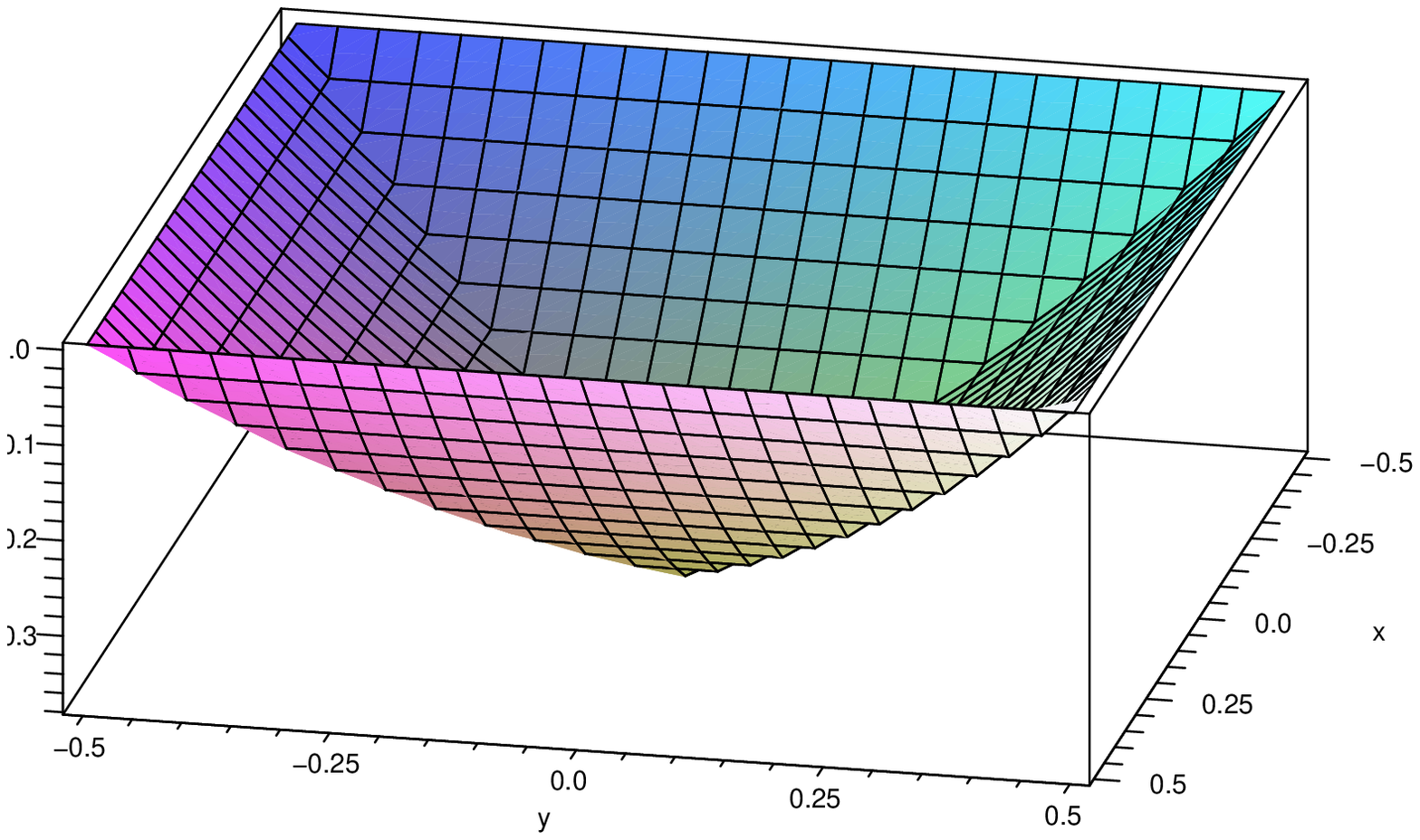}}%\hspace*{5mm}
%\rput(0,0)
{\includegraphics[width=170pt,height=120pt]{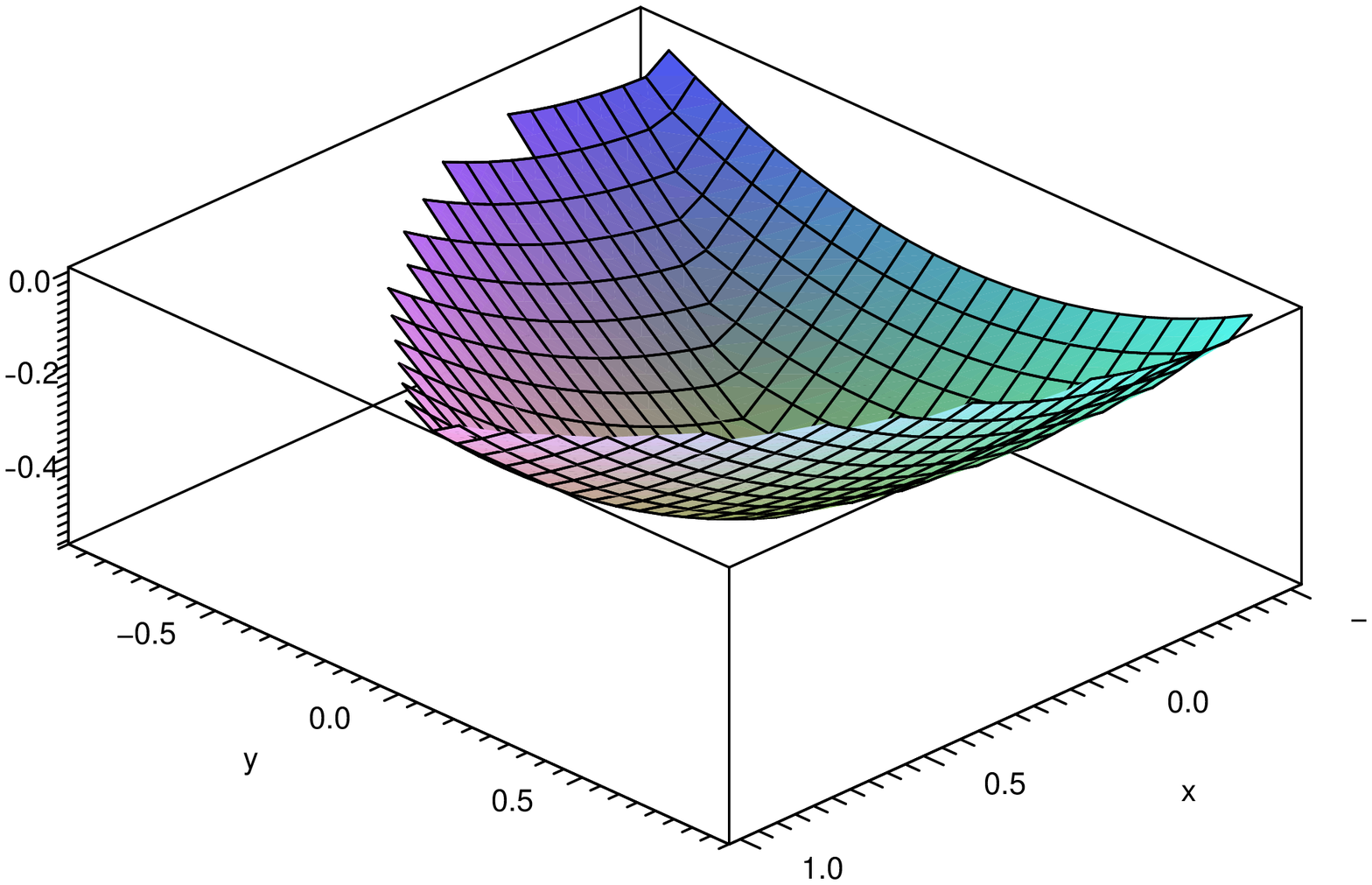}}\\
(a)\hskip 250pt (b)
\caption{Graphs of $u^a$ and $u^b$; a side view. The images are generated by Maple. The trajectory of a single particle of the flow is shown in each case.}
\label{fig_uaub}
\end{figure}

\rput(0,-0.3){
\psline[linewidth=0.8pt,linecolor=red,arrows=->,arrowscale=2](1.2,8.3)(1.2,6.1)
\psline[linewidth=0.8pt,linecolor=red,arrows=->,arrowscale=2](9.1,7.9)(9.1,5.3)
\psline[linewidth=0.8pt,linecolor=red,arrows=->,arrowscale=2](1.2,6.1)(5.5,6.4)
\psline[linewidth=0.8pt,linecolor=red,arrows=->,arrowscale=2](9.1,5.3)(12.8,5.5)
}

Until now it was not even known if $\inf_{\Om,u} R(u;\Om)$ is equal to or greater than $1/2$. It was found in \cite{CL3} that
$$
R(u^a;\Om^a) \approx 0.593
$$
for the function $u^a(x_1,x_2) = \max\{ \vphi(|x_1| + 1/2),\, \vphi(|x_2| + 1/2) \}$, with $\vphi(r) = (r^2-1)/2$ and $\Om^a = (-1/2,\, 1/2) \times (-1/2,\, 1/2)$. Besides, it was shown in \cite{comment} that
$$
R(u^b;\Om^b) \approx 0.581,
$$
where $u^b$ and $\Om^b$ are defined as follows. Take an equilateral triangle $ABC$ with unit sides and denote by $r_A(x),\, r_B(x),\, r_C(x)$ the distances from $x$ to $A,\, B,\, C$; then $u^b(x) = \max \{ \vphi(r_A(x)),\, \vphi(r_B(x)),\, \vphi(r_C(x))\}$ and $\Om^b = \{ x : r_A(x) < 1,\, r_B(x) < 1,\, r_C(x) < 1 \}$ is a Reuleaux triangle. The images of $u^a$ and $u^b$ generated by Maple are shown in Fig. \ref{fig_uaub}. The indentation on the boundary of the graph (b) is an artefact of the used computer program.

Thus, it was first found that $\inf_u R(u;\Om) < 0.594$, and this estimate was then substituted with a better one, $\inf_u R(u;\Om) < 0.582$. The function $u^a$ and then the function $u^b$ were for some time considered as true minimizers of the Comte\,-\,Lachand-Robert problem.

%\begin{remark}\label{z3}\rm
%One can state the relaxed problem, allowing for more than one reflection and taking, instead of the graph of a function, the boundary of an infinite domain containing the half-space $\{ z \le 0 \}$. In this setting formula (\ref{o Resistance}) for the resistance is no more valid; one needs to substitute it with an implicit formula. Nevertheless the solution for this relaxed problem is not difficult; it happens that the infimum of resistance is $0.5$. The solution is given in \cite{JDCS}.
%\end{remark}

Here we state our main results.

\begin{theorem}\label{t1}
$\inf_{u,\Om} R(u;\Om) = 1/2$.
\end{theorem}

As a simple corollary of Theorem \ref{t1}, we get

\begin{theorem}\label{t2}
For any domain $\Om$ one has $\inf_{u} R(u;\Om) = 1/2$.
\end{theorem}

These results are quite paradoxical; they mean that most part of the graph of $u$ should be formed by "mirrors" with the angle of inclination close to $45^0$. After the reflection from these mirrors the particles should move a very long way along gently sloping "valleys" below the "zero level", and the total area of these valleys should be close to zero. In the next section we give a solution which, in a part, is close to another paradoxical solution given by Besicovitch to the Kakeya problem \cite{Bes}.

\section{Proof of the main results}

Recall that the Kakeya problem is as follows: among all domains in which a segment of unit length can be continuously turned around through $360^0$, find the domain of smallest area. Besicovitch proved that the infimum of area is zero. An essential part of his proof was the following: he divided an isosceles triangle into a large number of small triangles by joining the apex with points on the base and then shifted the obtained triangles in the directions parallel to the base. The shift length was of course different for different triangles. As a result, a figure of arbitrarily small area was obtained.

Our task is somewhat more difficult, as will be seen below. We also divide an isosceles triangle into a large number of thin triangles and apply a linear transformation to each of them (different transformations for different triangles). Certain parts of the thin triangles should overlap to form a figure of vanishing relative area, and the remaining parts should be disjoint.

Let us first introduce some notation. Take a triangle $ABC$ and draw a segment $MN$ parallel to its base $AC$ with the endpoints on the two lateral sides (see Fig. \ref{fig_doubetriang}). The triangle $ABC$ and the segment $MN$ will be called the {\it big triangle} and the {\it separating segment}, respectively. The open sets $MBN$ and $AMNC$ will be called the {\it small triangle} and the {\it trapezoid} associated with the big triangle $ABC$ and denoted by the signs $\triangle$ and $\trap$. The ratio
$$
\kappa = \frac{\text{dist}(B,\, MN)}{\max\{ |AB|,\, |BC| \}}
$$
will be called {\it the ratio associated with the big triangle} $ABC$, and $\max\{ |AB|,\, |BC| \}$ will be denoted by $r_0$. Here of course $\text{dist}(B,\, MN)$ means the smallest distance from $B$ to a point of the segment $MN$. In Fig. \ref{fig_doubetriang}, the ratio equals $\kappa = |BH|/|BC|$, where $BH$ is the height of the small triangle, and $r_0 = |BC|$.

\begin{figure}[h]
\begin{picture}(0,160)
\rput(5.67,0.35){
\scalebox{1.15}{
\pspolygon(0,0)(2,0)(0.8,4)
\psline(0.2,1)(1.7,1)
\psline[linewidth=0.4pt,linestyle=dashed](0.8,4)(0.8,1)
\rput(-0.22,-0.1){\scalebox{0.87}{$A$}}
\rput(2.2,-0.1){\scalebox{0.87}{$C$}}
\rput(1.02,4.12){\scalebox{0.87}{$B$}}
\rput(-0.08,1){\scalebox{0.87}{$M$}}
\rput(1.97,1){\scalebox{0.87}{$N$}}
\rput(0.77,0.73){\scalebox{0.87}{$H$}}
\psline[linewidth=0.4pt,linecolor=blue,linestyle=dashed](0.7565,4.5)(1.2,-0.6)
\rput(1.29,0.2){\scalebox{0.77}{$P$}}
\rput(1.22,1.2){\scalebox{0.77}{$Q$}}
}
}
\end{picture}
\caption{A big triangle.}
\label{fig_doubetriang}
\end{figure}
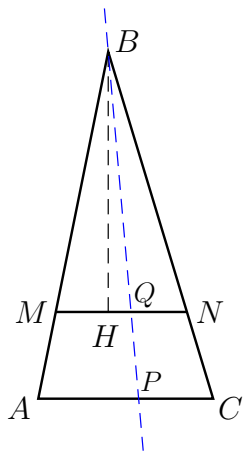

The values $\kappa$ and $r_0$ can also be characterized as follows: the smallest ring centered at $B$ that contains $\trap$ has the outer radius $r_0$ and the inner radius $\kappa r_0$.

Introduce polar coordinates on the $x$-plane $r=r(x),\, \theta=\theta(x)$ with the pole at $B$ and define the function $u_{ABC}$ in the closed domain $ABC$ by
\beq\label{o rule}
u_{ABC}(x) = \left\{ \begin{array}{ll} \frac{r^2(x) - r_0^2}{2r_0}, & \text{if } x \in \trap\\
-c, & \text{if } x \in \triangle \cup MN\\
\ 0, & \text{if } x \in AB \cup BC \cup CA
 \end{array} \right. .
\eeq
The positive constant $c$ is chosen so as
$$
-c \le \inf_{x\in\scalebox{0.75}{\trap}} u_{ABC}(x);
$$
thereby the function $u_{ABC}$ is not uniquely defined. Te function is negative in the interior of the triangle $ABC$ and is zero on its sides.

The graph of the restriction of $u_{ABC}$ on $\trap$ is a piece of a circular paraboloid with vertical axis and with focus at $(B,0)$. This means that a particle of the flow reflected from this piece of paraboloid will then move along a ray through the focus. Thus, the third coordinate of the reflected particle will gradually increase and therefore no further reflections will happen (the trajectory only touches the graph of $u_{ABC}$ at the point$(B,0)$). Further, a particle hitting the graph at a point corresponding to the small triangle $MBN$ (where $u_{ABC}$ is constant) is reflected back in the vertical direction and does not make reflections anymore. Thus, the function $u_{ABC}$ satisfies SIC.

This can also be checked in a purely analytical way. Indeed, if $x \in \trap$ then the vector $\nabla u_{ABC}(x)$ is proportional to $\overrightarrow{Bx}$; besides, $|\nabla u_{ABC}(x)|^2 = r^2(x)/r_0^2$, and the point $x - t\nabla u_{ABC}(x)$ lies on the segment $[B,\, x]$. One easily sees that $u(x - t\nabla u_{ABC}(x)) < u_{ABC}(x)$ (here one should consider the two cases where $x \in \trap$ and $x \in \triangle$), and therefore the left hand side of (\ref{o SIC}) is negative, while the right hand side is positive. If, otherwise, $x \in \triangle$ then $\nabla u_{ABC}(x) = 0$, and therefore the left hand side of (\ref{o SIC}) equals zero and the right hand side equals $1/2$.

A section of the graph of $u_{ABC}$ by the vertical plane through a horizontal straight line containing $B$ (the line $PB$ in Fig. \ref{fig_doubetriang}) is shown in Fig. \ref{fig_section}. The section is formed by an arc of parabola and a horizontal segment. Each particle of the flow that initially belongs to the plane of section, after reflection from the arc or the segment will also belong to this plane. A particle that hits the arc of parabola will then pass through $B$ (see Fig. \ref{fig_section}). One can also see that $|PB| < r_0$ and $|QB| > \kappa r_0$. In Fig. \ref{fig_section} the smallest possible value of the constant, $c = -\inf_{x\in\scalebox{0.75}{\trap}} u_{ABC}(x)$, has been chosen.

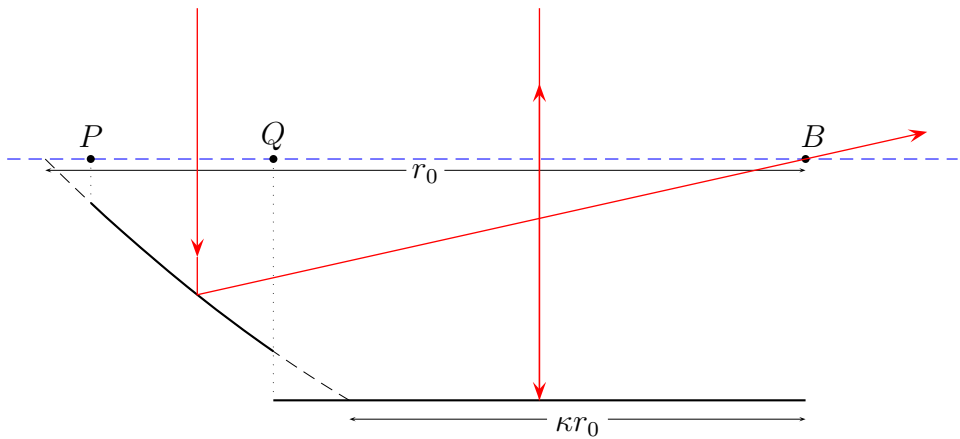
\begin{figure}[h]
\begin{picture}(0,170)
\rput(11,3.8){

\psecurve[linewidth=0.2pt,linestyle=dashed](-11,1.05)(-10,0)(-9,-0.95)(-8,-1.8)(-7,-2.55)(-6,-3.2)(-5,-3.75)
\psecurve[linewidth=0.8pt](-10,0)(-9.4,-0.582)(-9,-0.95)(-8,-1.8)(-7,-2.55)(-6,-3.2)
\psline[linewidth=0.4pt,linecolor=blue,linestyle=dashed](-10.5,0)(2,0)
\psline[linewidth=0.4pt,linestyle=dotted](-9.4,0)(-9.4,-0.582)
\psline[linewidth=0.4pt,linestyle=dotted](-7,0)(-7,-3.2)
\psline[linewidth=0.8pt](-7,-3.2)(0,-3.2)
\rput(0.1,0.3){$B$}
\rput(-9.4,0.3){$P$}
\rput(-7,0.3){$Q$}
\psdots[dotsize=3pt](0,0)(-9.4,0)(-7,0)
\psline[linewidth=0.1pt]{<-}(-6,-3.45)(-3.4,-3.45)
\psline[linewidth=0.1pt]{->}(-2.6,-3.45)(0,-3.45)
\rput(-3,-3.55){$\kappa r_0$}
\psline[linewidth=0.1pt]{<-}(-10,-0.15)(-5.25,-0.15)
\psline[linewidth=0.1pt]{->}(-4.75,-0.15)(0,-0.15)
\rput(-5,-0.2){$r_0$}
\psline[linewidth=0.5pt,linecolor=red,arrows=->,arrowscale=2](-8,2)(-8,-1.3)
\psline[linewidth=0.5pt,linecolor=red,arrows=->,arrowscale=2](-8,-1.3)(-8,-1.8)(1.6,0.36)
\psline[linewidth=0.5pt,linecolor=red,arrows=->,arrowscale=2](-3.5,2)(-3.5,-3.2)
\psline[linewidth=0.5pt,linecolor=red,arrows=->,arrowscale=2](-3.5,-3.2)(-3.5,1)
}
\end{picture}
\caption{A vertical section of the graph of $u_{ABC}$.}
\label{fig_section}
\end{figure}

One can now estimate the resistance associated with the trapezoid,
$$
R(u_{ABC}\rfloor_{\scalebox{0.75}{\trap}};\, \trap) = \frac{1}{|\trap|} \int_{\scalebox{0.75}{\trap}} \frac{dx}{1 + \frac{r^2(x)}{r_0^2}}.
$$
Taking into account that $r(x)/r_0 \ge \kappa$, one obtains that
$$
R(u_{ABC}\rfloor_{\scalebox{0.75}{\trap}};\, \trap) \le \frac{1}{1 + \kappa^2}.
$$
On the other hand, the resistance of the small triangle equals 1.

Intuitively, if $\kappa$ is close to to 1, the slope of motion of reflected particles will be small and therefore the resistance will be close to $1/2$. However, in this case the relative area of the small triangle (and therefore the resistance of the big triangle $ABC$) will be close to 1. The idea of the proof is to take a large collection of big triangles such that the associated small triangles effectively overlap, and so the relative area of their union is small.

For each natural $n$ we take a family of $2^n$ big triangles enumerated by $k = 1,\, 2,\, 3, \ldots, 2^n$. Let $\kappa_k^n$ be the associated ratios, $\triangle_k^{\!n}$ and $\trap_k^{\!n}$ be the corresponding small triangles and trapezoids, and denote
$$
\triangle^{\!n} = \cup_{k=1}^{2^n} \triangle_k^{\!n} \quad \text{and} \quad \trap^{\!n} = \cup_{k=1}^{2^n} \trap_k^{\!n}.
$$

\begin{propo}\label{propo1}
There exist families of triangles $\triangle_k^{\!n}$,\, $n = 1,\, 2, \ldots$;\, $k = 1, \ldots, 2^n$ satisfying the following conditions:

(i) for each $n$, the sets $\,\trap_1^{\!n},\ \,\trap_2^{\!n}, \ldots, \,\trap_{2^n}^{\!n},\ \triangle^{\!n}$ are mutually disjoint;
$$
\text{(ii)} \hspace{2mm} \lim_{n\to\infty} \frac{| \triangle^{\!n}|}{| \trap^{\!n}|} =0; \hspace*{90mm}
$$

(iii) there exists a sequence $a_n > 0$ converging to zero as $n \to \infty$ such that $\kappa_k^n \ge 1 - a_n$.
\end{propo}

This proposition is a key point of the proof. At the first glance it looks paradoxical: according to (iii), the area of each trapezoid $\trap_k^{\!n}$ is much smaller than the area of $\triangle_k^{\!n}$. On the other hand, (ii) implies that the small triangles $\triangle_k^{\!n}$,\, $k = 1, \ldots, 2^n$ strongly overlap, so that the area of their union is much smaller than the area of the union of trapezoids $\trap_k^{\!n}$.

\begin{proof}
First we define the procedure of $\del$-doubling. Take a big triangle $ABC$ with the separating segment $MN$, and let $T$ be the midpoint of $MN$ (see Fig. \ref{fig_doubling}). Denote $|MN| = a$, and let the height of the small triangle $MBN$ be $h$ and the height of the trapezoid $AMNC$ be $d$. Extend the sides $MB$ and $NB$ beyond the point $B$ to obtain the segments $MM'$ and $NN'$, with
$$
|BM'| = \del |BM| \quad \text{and} \quad |BN'| = \del |BN|.
$$
Let the straight lines $M'T$ and $N'T$ intersect the segment $AC$ at the points $C'$ and $A'$, respectively; $N'T$ intersects $MB$ at the point $R$, and $M'T$ intersects $NB$ at the point $S$. The procedure of $\del$-doubling applied to $ABC$ results in the two new big triangles $AM'C'$ and $A'N'C$; their separating lines $MT$ and $TN$ have the length $a/2$ each. The heights of the new small triangles $MM'T$ and $NN'T$ have the same length $(1 + \del)h$. The two obtained trapezoids do not intersect and belong to the original trapezoid, and the area of their union is greater than $ad$.

\begin{figure}[h]
\begin{picture}(0,220)
\rput(5.7,0.3){
\pspolygon(0,0)(3,0)(1,5)
\psline(0.2,1)(2.6,1)
\psline[linewidth=0.4pt,linecolor=blue,linestyle=dashed](0.2,7)(1,5)(1.4,7)
\psline[linewidth=0.4pt,linestyle=dotted](0.2,7)(1.4,7)
\psline[linewidth=0.4pt,linecolor=blue,linestyle=dashed](0.2,7)(1.6,0)
\psline[linewidth=0.4pt,linecolor=blue,linestyle=dashed](1.4,7)(1.4,0)
\psdots[dotsize=2.5pt](1.4,1)(0.8,4)(1.4,4)(1.4,0)(1.6,0)(0.2,7)(1.4,7)
\rput(-0.25,-0.1){$A$}
\rput(3.22,-0.1){$C$}
\rput(1.16,5.2){\scalebox{0.82}{$B$}}
\rput(-0.03,1.1){$M$}
\rput(2.8,1.16){$N$}
\rput(-0.07,7){\scalebox{0.9}{$N'$}}
\rput(1.68,7){\scalebox{0.82}{$M'$}}
\rput(1.22,-0.24){\scalebox{0.9}{$C'$}}
\rput(1.75,-0.22){\scalebox{0.82}{$A'$}}
\rput(1.2,0.73){\scalebox{0.9}{$T$}}
\psline[linewidth=0.4pt,linestyle=dotted](0.8,4)(1.4,4)
\rput(0.52,4){\scalebox{0.9}{$R$}}
\rput(1.63,4.05){\scalebox{0.9}{$S$}}
\psline[linewidth=0.4pt,linecolor=blue,linestyle=dashed](0.6,5)(1.4,5)
\rput(0.35,5){\scalebox{0.8}{$P$}}
\rput(1.65,5){\scalebox{0.8}{$Q$}}
}
\end{picture}
\caption{The procedure of doubling: the triangle $ABC$ is substituted with the triangles $AM'C'$ and $A'N'C$.}
\label{fig_doubling}
\end{figure}
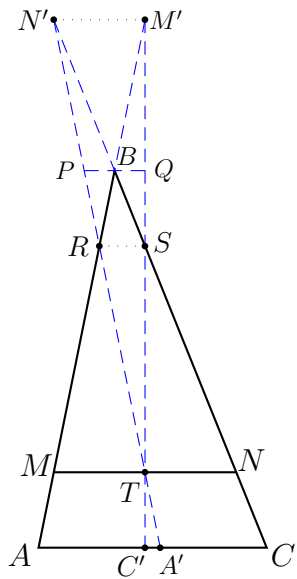

Draw the line through $B$ parallel to $MN$ and denote by $P$ and $Q$ the points of its intersection with $N'T$ and $M'T$. The triangles $TN'N$ and $PN'B$ are similar with the ratio $\del/(1+\del)$; therefore
$$
|PB| = \frac{\del}{1 + \del}\, |TN| = \frac{\del}{1 + \del}\, |MT|.
$$
Thus, the triangles $RPB$ and $RMT$ are similar with the same ratio; in particular, we have $|RB| = \del |MR|/(1+\del)$, hence
\beq\label{ratio}
|RB| = \frac{\del}{1 + 2\del}\, |MB|.
\eeq

Consider the triangles $MBN$ and $RBN'$. Relation (\ref{ratio}) gives the ratio of their sides $RB$ and $MB$, and the ratio of their heights dropped to these sides equals $\del$. Therefore the area of the triangle $RBN'$ equals
$$
|RBN'| = \frac{\del^2}{1 + 2\del}\, |MBN| = \frac{\del^2}{1 + 2\del} \cdot \frac{ah}{2}.
$$
The area of the triangle $SBM'$ is the same. Thus, the increase of the total area as a result of doubling is smaller than $\del^2 ah$.

Let us now apply the procedure of doubling several times. Initially one has the triangle $ABC$, and at the $m$th step ($m \ge 1$) one applies the procedure of $\del_m$-doubling to each of $2^{m-1}$ triangles obtained at the previous step. Thus, the separating lines of the triangles at the $m$th step have the length $2^{-m} a$ each, and their union is the segment $MN$. The height of each such triangle equals
$$
h_m = (1+\del_m)\cdot \ldots \cdot (1+\del_1)h,
$$
and its area equals $2^{-m-1} a\cdot h_m$.

Let $S_m$ be the area of the union of triangles at the $m$th step. The increase of the area at the $(m+1)$th step is smaller than $2^m \cdot \del_{m+1}^2 2^{-m} a h_m$; that is,
\begin{equation}\label{increaseS}
S_{m+1} < S_m + \del_{m+1}^2\, a h_m.
\end{equation}
Take $h_m = m+1$ (and in particular, $h = h_0 = 1$); then we have $\del_m = 1/m$ and $S_0 = a/2$, and by (\ref{increaseS}),
$$
S_{m+1} < S_m + \frac{a}{m+1}.
$$
One easily concludes that $S_m < a(\ln m + 3/2)$ for $m \ge 1$.

In Fig. \ref{fig_3steps} the initial triangle and the triangles obtained at the steps 1 -- 3 of the doubling procedure are shown.

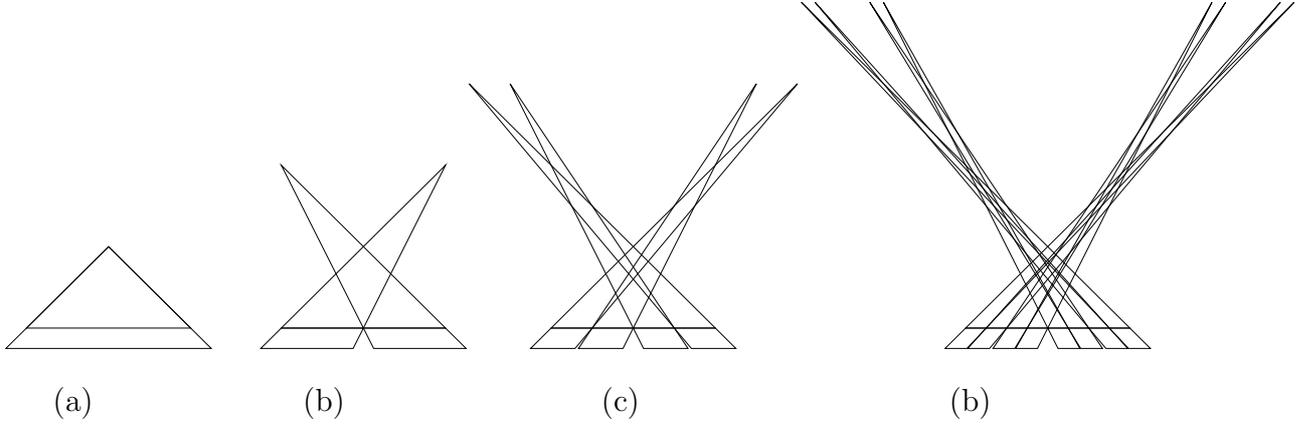
\begin{figure}[h]
\begin{picture}(0,170)
\rput(0,1.25){
\rput(0.25,0){
\scalebox{0.27}{
\pspolygon(-4,0)(4,0)(0,4)
\pspolygon(-4,0)(-5,-1)(5,-1)(0,4)
}
}
\rput(3.6,0){
\scalebox{0.27}{
\pspolygon(-4,0)(4,8)(0,0)
\pspolygon(4,0)(-4,8)(0,0)
\pspolygon(-4,0)(-5,-1)(-0.5,-1)(0,0)
\pspolygon(4,0)(5,-1)(0.5,-1)(0,0)
}
}
\rput(7.15,0){
\scalebox{0.27}{
\pspolygon(-4,0)(8,12)(-2,0)
\pspolygon(-2,0)(6,12)(0,0)
\pspolygon(4,0)(-8,12)(2,0)
\pspolygon(2,0)(-6,12)(0,0)
\pspolygon(-4,0)(-5,-1)(-2.833,-1)(-2,0)
\pspolygon(-2,0)(-2.667,-1)(-0.5,-1)(0,0)
\pspolygon(4,0)(5,-1)(2.833,-1)(2,0)
\pspolygon(2,0)(2.667,-1)(0.5,-1)(0,0)
}
}
\rput(12.6,0){
\scalebox{0.27}{
\pspolygon(4,0)(-12,16)(3,0)
\pspolygon(3,0)(-11.333,16)(2,0)
\pspolygon(-2,0)(8.667,16)(-1,0)
\pspolygon(-1,0)(8,16)(0,0)
\pspolygon(2,0)(-8.667,16)(1,0)
\pspolygon(1,0)(-8,16)(0,0)
\pspolygon(-4,0)(12,16)(-3,0)
\pspolygon(-3,0)(11.333,16)(-2,0)
\pspolygon(5,-1)(4,0)(3,0)(3.9375,-1)
\pspolygon(-5,-1)(-4,0)(-3,0)(-3.9375,-1)
\pspolygon(3.8958,-1)(3,0)(2,0)(2.833,-1)
\pspolygon(-3.8958,-1)(-3,0)(-2,0)(-2.833,-1)
\pspolygon(-2.667,-1)(-2,0)(-1,0)(-1.604,-1)
\pspolygon(2.667,-1)(2,0)(1,0)(1.604,-1)
\pspolygon(-1.5625,-1)(-1,0)(0,0)(-0.5,-1)
\pspolygon(1.5625,-1)(1,0)(0,0)(0.5,-1)
}
}
\rput(-0.2,-0.99){(a)}
\rput(3.1,-0.99){(b)}
\rput(7,-0.99){(c)}
\rput(11.6,-0.99){(b)}
}
\end{picture}
\caption{The original triangle (a); two triangles obtained at the 1st step (b); four triangles obtained at the 2nd step (c); eight triangles obtained at the 3rd step (d).}
\label{fig_3steps}
\end{figure}

For all $m$, the trapezoids of the $m$th step are disjoint. Indeed, let this be true at the $m$th step. Two trapezoids obtained when doubling a triangle of the $m$th step are disjoint and are contained in the trapezoid associated with the original triangle; therefore they do not intersect any other trapezoid obtained at the $(m+1)$th step. The trapezoids also do not intersect the small triangles, since they lie on the opposite sides of the line $MN$. The total area of the trapezoids is greater than $ad$.

Fix $n$ and take $d = \sqrt n$. Let $\triangle_k^{\!n}$ and $\trap_k^{\!n}$,\, $k = 1, \ldots, 2^n$ be the triangles and trapezoids of the $n$th step. We have already verified that (i) is true. Further,
$$
|\triangle^{\!n}| = |\cup_{k=1}^{2^n} \triangle_k^{\!n}| = S_n < a(\ln n + 3/2)
$$
and the total area of trapezoids is
$$
|\,\trap^{\!n}| = |\cup_{k=1}^{2^n} \trap_k^{\!n}| > a \sqrt n;
$$
therefore (ii) is also fulfilled.

Let now $A_kBC_k$ be the $k$th big triangle ($1 \le k \le 2^n$) and $M_kN_k$ be its separating line. Assume without loss of generality that $|A_kB| \ge |C_kB|$. Using that
$$
\frac{|M_kB|}{|A_kB|} = \frac{h_n}{h_n+d} = \frac{n+1}{n + 1 + \sqrt n}\,,
$$
$|M_kN_k| = 2^{-n} a$, and $|A_kB| > n+1 + \sqrt n$, one obtains
$$
\kappa^n_k = \frac{\text{dist}(B, M_kN_k)}{|A_kB|} \ge \frac{|M_kB| - |M_kN_k|}{|A_kB|} \ge \frac{n + 1}{n + 1 + \sqrt n} - \frac{2^{-n} a}{ n+1 + \sqrt n},
$$
so (iii) is also satisfied.

\end{proof}

Now we use relation (\ref{o rule}) to define the function $u_k^n$ in each big triangle, choosing the constant $c = c_n$ to be the same for all $k$. Let $\Om_n = \triangle^{\!n} \cup \trap^{\!n}$. Using (i), we define the function $u_n$ on $\bar\Om_n$ so that its restriction on each big triangle of the $n$th step coincides with the corresponding function $u_k^n$. Using now (ii) and (iii), we obtain the estimates for the resistance of $u_n$,
$$
R(u_n; \Om_n) \le \frac{|\trap^{\!n}|}{|\trap^{\!n}| + |\triangle^{\!n}|}\ \frac{1}{1 + (1-a_n)^2} + \frac{|\triangle^{\!n}|}{|\trap^{\!n}| + |\triangle^{\!n}|} \to 1/2 \quad \text{as} \ \, n \to \infty.
$$
Thus, $\inf_{u,\Om} R(u; \Om) = 1/2$. Theorem \ref{t1} is proved.

Let us now prove Theorem \ref{t2}.

We say that $\Om^i$ is a copy of $\Om$, if there exist a real value $k_i > 0$ and an isometry $f_i$ of the plane such that $f_i (k_i \Om) = \Om^i$.

\begin{propo}\label{propo2}
For any two bounded domains $\Om$ and $\tilde\Om$ there exists a finite family $\{ \Om^i \}$ of mutually non-intersecting copies of $\Om$, all contained in $\tilde\Om$, such that $|\tilde\Om \setminus (\cup_i \Om^i)| < \ve$.
\end{propo}

\begin{proof}
Take a square $|x_1| < M, \ |x_2| < M$ containing $\Om$ and fix $\del = |\Om|/M^2$. Obviously, $0 < \del < 1$, and any square $Q$ on the plane contains a copy of $\Om$ that occupies the area $\del |Q|$. Further, take a square lattice $x_1 = am, \ x_2 = an, \ a > 0, \ m,\, n \in \ZZZ$, choosing $a$ so small that the squares $Q^i$ of the lattice contained in the domain $\tilde\Om$ occupy more than one half of its area, $|\cup_i Q^i| > \frac 12\, |\tilde\Om|$. For each square $Q^i$ take a copy $\Om^{(i0)}$ of $\Om$ contained in $Q^i$ and such that $|\Om^{(i0)}| = \del|Q^i|$. Thus, we have
$$
|\cup_i \Om^{(i0)}| > \frac{\del}{2} |\tilde\Om|.
$$

Next we inductively define the sequence of domains $\tilde\Om^j$,\, $j = 0,\, 1, \ldots, j_0$ and finite families $\Om^{(ij)} \ (j = 1,\ldots,j_0-1)$ of copies of $\Om$ such that for all $j$ the domains of the family $\{ \Om^{(ij)} \}_i$ are mutually disjoint,
$$
|\cup_i \Om^{(ij)}| > \frac{\del}{2} |\tilde\Om^j|, \quad \tilde\Om^0 = \tilde\Om, \quad \tilde\Om^{j+1} = \tilde\Om^j \setminus (\cup_i \Om^{(ij)}) \ \ \text{for}\, \ 0 \le j \le j_0 - 1,
$$
and $(1-\del/2)^{j_0} < \ve$. We have $|\tilde\Om^j| < (1-\del/2)^{j} |\tilde\Om|$, and therefore,
$$
|\Om \setminus (\cup_{i,j} \Om^{(ij)})| = |\tilde\Om^{j_0}| < \ve.
$$
\end{proof}

\begin{propo}\label{propo3}
For any pair of bounded domains $\Om$,\, $\tilde\Om$ and any admissible function $u : \bar\Om \to \RRR$ there exists an admissible function $\tilde u : \overline{\tilde\Om} \to \RRR$ such that 
$$
R(\tilde\Om; \tilde u) < R(\Om; u) + \ve.
$$
\end{propo}

\begin{proof}
Take a finite family $\Om^i = f_i(k_i \Om)$ of non-intersecting copies of $\Om$ contained in $\tilde\Om$ and such that $|\tilde\Om \setminus (\cup_i \Om^i)| < \ve |\tilde\Om|$. The transformations $f_i k_i$ induce admissible functions $u_i$ on $\overline{\Om^i}$ by 
$$
u_i(f_i(k_i x)) = k_i u(x) \quad \text{for all} \ \, x \in \bar\Om.
$$
One easily verifies that $R(u_i;\Om^i) = R(u;\Om)$. 

Define the admissible function $\tilde u$ on $\overline{\tilde\Om}$ by
$$
\tilde u(x) = \left\{ 
\begin{array}{ccl}
u_i(x), & \text{if} & x \in \Om^i\\
-c, & \text{if} & x \in \tilde\Om \setminus (\cup_i \Om^i)\\
0, & \text{if} & x \in \pl\tilde\Om,
\end{array}
\right.
$$
where $c$ is an arbitrary positive constant. One has
$$
R(\tilde\Om; \tilde u) = \sum_i \frac{|\Om^i|}{|\tilde\Om|}\, R(u_i;\Om^i) + \frac{|\tilde\Om \setminus (\cup_i \Om^i)|}{|\tilde\Om|} < R(\Om; u) + \ve.
$$
\end{proof}

By Theorem \ref{t1}, there exists a sequence of admissible functions $u_n : \bar\Om_n \to \RRR$ such that $R(u_n;\Om_n) \to 1/2$. Let $\Om$ be a bounded domain. By Proposition~\ref{propo3}, for each natural $n$ there exists an admissible function $\tilde u_n : \bar\Om \to \RRR$ such that $R(\tilde u_n;\Om)< R(u_n;\Om_n) + 1/n$. This implies that $\inf_{u} R(u;\Om) = 1/2$. Theorem \ref{t2} is proved.

\section*{Acknowledgements}

This work was supported in part by FEDER funds through COMPETE – Operational Programme Factors of Competitiveness ("Programa Operacional Factores de Competitividade") and by Portuguese funds through the Center for Research and Development in Mathematics and Applications and the Portuguese Foundation for Science and Technology ("FCT-Funda\c{c}\~ao para a Ci\^encia e a Tecnologia"), within project PEst-C/MAT/UI4106/2011 with COMPETE number FCOMP-01-0124-FEDER-022690, as well as by the FCT research project PTDC/MAT/113470/2009.

\end{document}